\documentclass[12pt]{amsart}

\usepackage{latexsym}
\usepackage{amssymb}
\usepackage{amsmath}
\usepackage{mathrsfs}

\newtheorem{theorem}{Theorem}[section]

\newtheorem{proposition}[theorem]{Proposition}

\theoremstyle{definition}
\newtheorem{definition}[theorem]{Definition}

\theoremstyle{remark}
\newtheorem{remark}[theorem]{Remark}

\newcommand{\F}{\mathbb{F}}
\newcommand{\N}{\mathbb{N}}

\newcommand{\Q}{\mathbb{Q}}
\newcommand{\R}{\mathbb{R}}

\newcommand{\p}{\mathcal{P}}

\newcommand{\num}{\mathfrak{n}}

\newcommand{\ring}{\mathfrak}
\newcommand{\A}{\mathfrak{A}}
\newcommand{\B}{\mathfrak{B}}
\newcommand{\C}{\mathfrak{C}}

\newcommand{\sh}{\text{sh}}
\newcommand{\st}{\text{st}}

\begin{document}

\title{Some applications of numerosities in measure theory}

\author{Vieri Benci}
\author{Emanuele Bottazzi}
\author{Mauro Di Nasso}

\address{Dipartimento di Matematica,
Universit\`{a} di Pisa, Italy and Department of Mathematics, College of Science, King Saud University, Riyadh,  Saudi Arabia}
\email{benci@dma.unipi.it}

\address{Dipartimento di Matematica,
Universit\`{a} di Trento, Italy.}
\email{emanuele.bottazzi@unitn.it}

\address{Dipartimento di Matematica,
Universit\`{a} di Pisa, Italy.}
\email{dinasso@dm.unipi.it}


\maketitle

\begin{abstract}
We present some applications of the notion of numerosity
to measure theory, including the construction of a
non-Archimedean model for the probability
of infinite sequences of coin tosses.
\end{abstract}

\section*{Introduction}

The idea of numerosity as a notion of measure for the size of
infinite sets was introduced by the first named author
in \cite{be95}, and then given sound logical foundations
in \cite{bd}. A theory of numerosities have been then developed
in a sequel of papers (see, \emph{e.g.}, \cite{bdf,df}).
The main feature of numerosities is that they satisfy the
same basic formal properties as finite cardinalities,
including the fact that proper subsets must have strictly
smaller sizes. This has to be contrasted with Cantorian
cardinalities, where every infinite set have proper subsets
of the same cardinality.

\smallskip
In this paper we will present three applications of numerosity
in topics of measure theory.
The first one is about the existence of ``inner measures''
associated to any given non-atomic pre-measure.
The second application is focused on sets of real numbers.
We show that elementary numerosities
provide a useful tool with really strong compatibility
properties with respect to the Lebesgue measure.
For instance, intervals of equal length can be given the same
numerosity, and 
any interval of rational length $p/q$ has a numerosity which 
is exactly  $p/q$.
We derive consequences about
the existence of totally defined finitely additive measures
that extend the Lebesgue measure.
Finally, the third application is about non-Archimedean probability.
Following ideas from \cite{nap}, we consider a
model for infinite sequences of coin tosses which is
coherent with the original view of Laplace.
Indeed, probability of an event is defined as
the numerosity of positive outcomes divided by the
numerosity of all possible outcomes;  moreover, the probability of
cylindrical sets exactly coincides with the usual Kolmogorov probability.

\section{Terminology and preliminary notions}

We fix here our terminology, and recall a few basic facts from measure
theory and numerosity theory
that will be used in the sequel.

\smallskip
Let us first agree on notation.
We write $A\subseteq B$ to mean that $A$ is a subset of $B$,
and we write $A\subset B$ to mean that $A$ is a \emph{proper} subset of $B$.
The \emph{complement} of a set $A$ is denoted by $A^c$, and its \emph{powerset}
is denoted by $\p(A)$. We use the symbol $\sqcup$ to denote
\emph{disjoint unions}.
By $\N$ we denote the set of \emph{positive integers}.
For an ordered field $\F$, we denote by
$[0,\infty)_\F=\{x\in\F\mid x\ge 0\}$
the set of its non-negative elements.
We will write $[0,+\infty]_\R$ to denote the set of non-negative
real numbers plus the symbol $+\infty$, where we agree
that $x+\infty=+\infty+x=+\infty+\infty=+\infty$ for all $x\in\R$.

\medskip
\begin{definition}
A \emph{finitely additive measure} is a triple $(\Omega,\A,\mu)$ where:

\smallskip
\begin{itemize}
\item
The \emph{space} $\Omega$ is a nonempty set;

\smallskip
\item
$\A$ is an \emph{algebra of sets} over $\Omega$, \emph{i.e.}
a nonempty family of subsets of $\Omega$ which is closed under
finite unions and intersections, and under relative complements, \emph{i.e.}
$A,B\in\A\Rightarrow A\cup B, A\cap B, A\setminus B\in\A$;\footnote
{~Actually, the closure under intersections follow from the other two properties,
since $A\cap B=A\setminus(A\setminus B)$.}

\smallskip
\item
$\mu:\A\to [0,+\infty]_\R$ is an \emph{additive function}, \emph{i.e.}
$\mu(A\cup B)=\mu(A)+\mu(B)$ whenever $A,B\in\A$ are disjoint.\footnote
{~Such functions $\mu$ are sometimes called \emph{contents} in the literature.}
We also assume that $\mu(\emptyset)=0$.
\end{itemize}

\smallskip
The measure $(\Omega,\A,\mu)$ is called \emph{non-atomic} when all
finite sets in $\A$ have measure zero. We say that
$(\Omega,\A,\mu)$ is a \emph{probability measure} when $\mu:\A\to[0,1]_\R$
takes values in the unit interval, and $\mu(\Omega)=1$.

\end{definition}

\medskip
For simplicity, in the following we will often identify the
triple $(\Omega,\A,\mu)$ with the function $\mu$.

\smallskip
Remark that a finitely additive measure $\mu$ is
necessarily \emph{monotone}, \emph{i.e.}

\smallskip
\begin{itemize}
\item
$\mu(A)\le\mu(B)$ for all $A,B\in\A$ with $A\subseteq B$.
\end{itemize}

\medskip
\begin{definition}
A finitely additive measure $\mu$ defined on a ring of sets $\A$ is
called a \emph{pre-measure} if it is $\sigma$-\emph{additive},
\emph{i.e.} if for every countable family $\{A_n\}_{n\in\N}\subseteq\A$
of pairwise disjoint sets whose union lies in $\A$, it holds:
$$\mu\left(\bigsqcup_{n\in\N}A_n\right)\ =\ \sum_{n=1}^\infty\mu(A_n).$$
A \emph{measure} is a pre-measure which is defined on a $\sigma$-\emph{algebra},
\emph{i.e.} on an algebra of sets which is closed under
countable unions and intersections.
\end{definition}

\medskip
\begin{definition}
An \emph{outer measure} on a set $\Omega$ is a function
$$M:\p(\Omega)\to[0,+\infty]_\R$$
defined on all subsets of $\Omega$ which is \emph{monotone}
and $\sigma$-\emph{subadditive}, \emph{i.e.}
$$M\left(\bigcup_{n\in\N}A_n\right)\ \leq\ \sum_{n\in\N}M(A_n).$$
It is also assumed that $M(\emptyset)=0$.
\end{definition}

\medskip
\begin{definition}
Given an outer measure $M$ on $\Omega$, the following family is
called the \emph{Caratheodory $\sigma$-algebra} associated to $M$:
$$\mathfrak{C}_M\ =\ \left\{X\subseteq\Omega\mid
M(Y)=M(X\cap Y)+M(X\setminus Y)\ \text{for all }Y\subseteq\Omega\right\}.$$
\end{definition}

\medskip
A well known theorem of Caratheodory states that the above family is
actually a $\sigma$-algebra, and that the restriction of $M$ to
$\mathfrak{C}_M$ is a \emph{complete} measure, \emph{i.e.}
a measure where $M(X)=0$ implies $Y\in\mathfrak{C}_M$ for all $Y\subseteq X$.
This result is usually combined with the property that every pre-measure
$\mu$ over a ring $\A$ of subsets of $\Omega$
is canonically extended to the outer measure
$\overline{\mu}:\p(\Omega)\to[0,\infty]_\R$ defined by putting:
$$\overline{\mu}(X)\ =\
\inf\left\{\sum_{n=1}^\infty\mu(A_n)\,\Big|\,
\{A_n\}_n\subseteq\A\ \&\ X\subseteq\bigcup_{n\in\N}A_n\right\}.$$

Indeed, a fundamental result in measure theory is that
the above function $\overline{\mu}$ is actually an outer measure
that extends $\mu$, and that the associated Caratheodory $\sigma$-algebra
$\mathfrak{C}_{\overline{\mu}}$ includes $\A$.
Moreover, such an outer measure $\overline{\mu}$ is \emph{regular},
\emph{i.e.} for all $X\in\p(\Omega)$ there exists $C\in\mathfrak{C}_{\overline{\mu}}$
such that $X\subseteq C$ and $\overline{\mu}(X)=\overline{\mu}(C)$.
(See \emph{e.g}. \cite{Yeh} Prop. 20.9.)

\medskip
Next, we will recall the notion of \emph{elementary numerosity},
a variant of the notion of numerosity that was introduced
in \cite{bbd}. The underlying idea is that of refining the notion
of finitely additive measure in such a way that also single points count.
To this end, one needs to consider ordered fields that extend the real line.

\smallskip
Recall that every ordered field $\F$ that properly extend $\R$
is necessarily \emph{non-Archimedean}, in that it
contains \emph{infinitesimal numbers} $\epsilon\ne 0$
such that $-1/n<\epsilon<1/n$ for all $n\in\N$.
Two elements $\xi,\zeta\in\F$ are called \emph{infinitely close}
if $\xi-\zeta$ is infinitesimal; in this case, we write
$\xi\approx\zeta$.
A number $\xi\in\F$ is called \emph{finite} if $-n<\xi<n$
for some $n\in\N$, and it is called \emph{infinite} otherwise.
Clearly, a number $\xi$ is infinite if and only if
its reciprocal $1/\xi$ is infinitesimal.
We remark that every finite $\xi\in\F$ is infinitely close to a
unique real number $r$, namely $r=\inf\{x\in\R\mid x>\xi\}$.
Such a number $r$ is called the \emph{standard part} of $\xi$,
and is denoted by $r=\st(\xi)$. Notice that $\st(\xi+\zeta)=\st(\xi)+\st(\zeta)$
and $\st(\xi\cdot\zeta)=\st(\xi)\cdot\st(\zeta)$ for all finite $\xi,\zeta$.
The notion of standard part can be extended to the infinite
elements $\xi\in\F$ by setting $\st(\xi)=+\infty$ when
$\xi$ is positive, and $\st(\xi)=-\infty$
when $\xi$ is negative.

\medskip
\begin{definition}
An \emph{elementary numerosity} on the set $\Omega$ is a function
$$\num:\p(\Omega)\longrightarrow[0,+\infty)_\F$$
defined on all subsets of $\Omega$, taking values
in an ordered field $\F\supseteq\R$ that extends the real line,
and that satisfies the following two properties:

\smallskip
\begin{enumerate}
\item
\emph{Additivity}:\ $\num(A\cup B)=\num(A)+\num(B)$ whenever
$A\cap B=\emptyset$;

\smallskip
\item
\emph{Unit size}:\ $\num(\{x\})=1$ for every point $x\in\Omega$.
\end{enumerate}
\end{definition}

\medskip
Notice that if $\Omega$ is a finite set,
then the only elementary numerosity is the finite cardinality.
On the other hand, when $\Omega$ is infinite, then
the numerosity function must also take ``infinite'' values,
and so the field $\mathbb{F}$ must be non-Archimedean.
It is worth remarking that also Cantorian cardinality
satisfies the above properties $(1),(2)$,
but the sum operation between cardinals is really far from
being a ring operation.\footnote
{~Recall that for infinite cardinals $\kappa,\nu$
it holds $\kappa+\nu=\max\{\kappa,\nu\}$.}

\smallskip
As straight consequences of the definition, we obtain
that elementary numerosities can be seen as generalizations of
finite cardinalities. Indeed, one can easily show that

\begin{itemize}
\item
$\num(A)=0$ if and only if $A=\emptyset$;
\item
If $A\subset B$ is a proper subset, then $\num(A)<\num(B)$;
\item
If $F$ is a finite set of cardinality $n$, then $\num(F)=n$.
\end{itemize}

\smallskip
Given an elementary numerosity and a ``measure unit" $\beta\in\F$,
there is a canonical way to construct a (real-valued)
finitely additive measure.

\medskip
\begin{definition}
If $\num:\p(\Omega)\to[0,+\infty)_\F$ is an elementary numerosity,
and $\beta\in\F$ is a positive number, the map
$\num_\beta:\p(\Omega)\to[0,+\infty]_\R$
is defined by setting
$$\num_\beta(A)\ =\ \sh\left(\frac{\num(A)}{\beta}\right).$$
\end{definition}

\medskip
\begin{proposition}
$\num_\beta$ is a finitely additive measure. Moreover,
$\num_\beta$ is non-atomic if and only if $\beta$ is an infinite number.
\end{proposition}

\begin{proof}
For all disjoint $A,B\subseteq\Omega$, one has:
\begin{eqnarray*}
\nonumber
\num_\beta(A\cup B) &=&
\st\left(\frac{\num(A\cup B)}{\beta}\right)\ =\
\st\left(\frac{\num(A)}{\beta}+\frac{\num(B)}{\beta}\right)
\\
\nonumber
{} &=&
\st\left(\frac{\num(A)}{\beta}\right)+\st\left(\frac{\num(B)}{\beta}\right)\ =\
\num_\beta(A)+\num_\beta(B).
\end{eqnarray*}

Notice that the measure $\num_\beta$ is non-atomic if and only if
$\num_\beta(\{x\})=\sh(1/\beta)=0$, and this
holds if and only if $\beta$ is infinite.
\end{proof}

\medskip
The relevant result about elementary numerosities
that we will use in the sequel,
is the following representation theorem, that
was proved in \cite{bbd}:

\medskip
\begin{theorem}\label{numeasure}
Let $(\Omega,\A,\mu)$ be a non-atomic finitely additive measure
on the infinite set $\Omega$, and let $\B\subseteq\A$ be a subalgebra
that does not contain nonempty null sets.
Then there exist

\smallskip
\begin{itemize}
\item
a non-Archimedean field $\F\supset\R$\,;

\smallskip
\item
an elementary numerosity $\num:\p(\Omega)\to[0,+\infty)_\F$\,;
\end{itemize}

\smallskip
\noindent
such that:

\smallskip
\begin{enumerate}
\item
$\mu(B)=\mu(B')\Leftrightarrow\num(B)=\num(B')$
for all $B,B'\in\B$ of finite measure;
\item
For every set $Z\in\A$ of positive finite measure,
if $\beta={\num(Z)}/{\mu(Z)}$ then $\mu(A)=\num_\beta(A)$
for all $A\in\A$.
\end{enumerate}
\end{theorem}

\section{Numerosities and inner measures}

In this section we will use elementary numerosities
to prove a general existence result about ``inner'' measures.

\medskip
\begin{theorem} \label{inner measure theorem}
Let $\A$ be an algebra of subsets of $\Omega$ and let
$\mu:\A\to[0,+\infty]_\R$ be a non-atomic pre-measure.
Assume that $\mu$ is non-trivial, in the sense that
there are sets $Z\in A$ with $0<\mu(Z)<+\infty$.
Then, along with the associated outer measure $\overline{\mu}$,
there exists an ``inner'' finitely additive measure
$$\underline{\mu}:\p(\Omega)\to[0,+\infty]_\R$$
such that:

\smallskip
\begin{enumerate}
\item
$\underline{\mu}(C)=\overline{\mu}(C)$ for all
$C\in\mathfrak{C}_\mu$, the Caratheodory $\sigma$-algebra associated to $\mu$.
In particular, $\underline{\mu}(A)=\mu(A)=\overline{\mu}(A)$ for all
$A\in\mathfrak{A}$.

\smallskip
\item
$\underline{\mu}(X)\le\overline{\mu}(X)$ for all
$X\subseteq\Omega$.
\end{enumerate}
\end{theorem}

\begin{proof}
By \emph{Caratheodory extension theorem}, the restriction
of $\overline{\mu}$ to $\mathfrak{C}_\mu$ is a measure
that agrees with $\mu$ on $\A$.
Now we apply Theorem \ref{numeasure} to the measure
$(\mathfrak{C}_\mu,\A,\overline{\mu})$, and
obtain the existence of an elementary numerosity
$\num:\p(\Omega)\to[0,+\infty)_\F$.
By property (2) in the Theorem,
if we pick any number $\beta=\frac{\num(Z)}{\mu(Z)}$ where
$0\mu(Z)<+\infty$, then $\num_\beta(C)=\overline{\mu}(C)$ for all
$C\in\mathfrak{C}_\mu$. We claim that
$\underline{\mu}=\num_\beta:\p(\Omega)\to[0,+\infty]_\R$
is the desired ``inner" finitely additive measure.

\smallskip
Property $(1)$ is trivially satisfied by our definition
of $\underline{\mu}$, so let us turn to $(2)$.
For every $X\subseteq\Omega$, by definition of outer measure we have that
for every $\epsilon>0$
there exists a countable union $A=\bigcup_{n=1}^\infty A_n$
of sets $A_n\in\A$ such that $A\supseteq X$ and
$\sum_{n=1}^\infty\mu(A_n)\le\overline{\mu}(X)+\epsilon$.
Notice that $A$ belongs to the $\sigma$-algebra generated by $\A$,
and hence $A\in\mathfrak{C}_\mu$. In consequence,
$\underline{\mu}(A)=\num_\beta(A)=\overline{\mu}(A)$.
Finally, by monotonicity of the finitely additive measure $\underline{\mu}$,
and by $\sigma$-subadditivity of the outer measure $\overline{\mu}$, we obtain:
$$\underline{\mu}(X)\ \le\ \underline{\mu}(A)\ =\ \overline{\mu}(A)\ \le\
\sum_{n=1}^\infty\overline{\mu}(A_n)\ =\ \sum_{n=1}^\infty\mu(A_n)\ \le\
\overline{\mu}(X)+\epsilon.$$
As $\epsilon>0$ is arbitrary, the desired inequality
$\underline{\mu}(X)\le\overline{\mu}(X)$ follows.
\end{proof}

\medskip
It seems of some interest to investigate the properties of the
extension of the Caratheodory algebra given by family of all sets
for which the outer measure coincides with the above ``inner measure":
$$\mathfrak{C}(\num_\beta)\ =\ \left\{X\subseteq\Omega\mid\underline{\mu}(X)=
\overline{\mu}(X)\right\}.$$
Clearly, the properties of $\mathfrak{C}(\num_\beta)$
may depend on the choice of the
elementary numerosity $\num$.

\smallskip
Theorem \ref{inner measure theorem} ensures that
the inclusion $\C_\mu \subseteq \C(\num_\beta)$ always holds.
Moreover, this inclusion is an equality if and only if all $X \not \in \C_\mu$
satisfy the inequality $\underline{\mu}(X) < \overline{\mu}(X)$.
It turns out that, when $\mu(\Omega) < +\infty$, this property is equivalent to a number of other
statements.

\medskip
\begin{proposition} \label{proposition equivalences}
If $\mu(\Omega) < + \infty$, then the following are equivalent:

\smallskip
\begin{enumerate}
\item
$\C_\mu=\C(\num_\beta)$.

\smallskip
\item
$X\not\in \C_\mu \Rightarrow \underline{\mu}(X) < \overline{\mu}(X)$
and $\underline{\mu}(X^c) < \overline{\mu}(X^c)$.

\smallskip
\item
$\underline{\mu}(X)=\overline{\mu}(X)\Longleftrightarrow\underline{\mu}(X^c)=
\overline{\mu}(X^c)$.

\smallskip
\item
$\underline{\mu}(X) = 0 \Longleftrightarrow \overline{\mu}(X) = 0$.
\end{enumerate}

\smallskip
\noindent
If $\mu(\Omega) = + \infty$, then
$(1)\Leftrightarrow (2) \Rightarrow (3) \Rightarrow (4)$.
\end{proposition}

\begin{proof}
We have already seen that $(1)$ and $(2)$ are equivalent.

\smallskip
$(2)\Rightarrow(3)$.
Suppose towards a contradiction that $(2)$ holds but $(3)$ is false.
The latter hypothesis ensures the existence of a set $X$
such that $\underline{\mu}(X)=\overline{\mu}(X)$ and
$\underline{\mu}(X^c)<\overline{\mu}(X^c)$.
Thanks to Theorem \ref{inner measure theorem}, we deduce that
$X\not\in\C_\mu$. By $(2)$ we get the contradiction
$\underline{\mu}(X)<\overline{\mu}(X)$.

\smallskip
$(3)\Rightarrow(4)$.
The implication $\overline{\mu}(X)=0\Rightarrow\underline{\mu}(X)=0$
is always true.
On the other hand, if $\underline{\mu}(X)=0$, then
$\underline{\mu}(X^c)=\underline{\mu}(\Omega)=\overline{\mu}(\Omega)$.
By the inequality $\underline{\mu}(X^c)\leq\overline{\mu}(X^c)$,
we deduce $\overline{\mu}(X^c)=\overline{\mu}(\Omega)=\underline{\mu}(X^c)$ and,
thanks to $(3)$, also $\overline{\mu}(X)=0$ follows.

\smallskip
$(4)\Rightarrow(2)$, under the hypothesis that $\mu(\Omega) < + \infty$.
Suppose towards a contradiction that $(4)$ holds but $(2)$ is false.
The latter hypothesis ensures the existence of a set $X\not\in\C_\mu$
satisfying $\underline{\mu}(X)=\overline{\mu}(X)$ and
$\underline{\mu}(X^c)<\overline{\mu}(X^c)$.
Thanks to Propositions 20.9 and 20.11 of \cite{Yeh},
we can find a set $A\in\mathfrak{C}_\mu$ satisfying
$A\supset X$, $\overline{\mu}(A)=\overline{\mu}(X)$ and
$\overline{\mu}(A\setminus X)> 0$.
From the hypothesis $\underline{\mu}(X)=\overline{\mu}(X)$ we
obtain the following equalities:
$$\underline{\mu}(X)\ =\ \overline{\mu}(X)\ =\
\overline{\mu}(A)\ =\ \underline{\mu}(A).$$
The above equalities and the hypothesis 
$\mu(\Omega) < + \infty$ imply $\underline{\mu}(A\setminus X)=0$.
By $(4)$, we obtain the contradiction
$\overline{\mu}(A\setminus X)=0$.
\end{proof}

\section{Numerosities and Lebesgue measure}

In this section, we show that elementary numerosities exist which
are consistent with Lebesgue measure in a strong sense.
Precisely, the following result holds:

\medskip
\begin{theorem}
Let $(\R,\mathfrak{L},\mu_L)$ be the \emph{Lebesgue measure} over $\R$.
Then there exists an elementary numerosity
$\num:\p(\R)\rightarrow[0,+\infty)_\F$ such that:

\smallskip
\begin{enumerate}
\item
$\num([x,x+a))=\num([y,y+a))$ for all $x,y\in\R$ and for all $a>0$.

\smallskip
\item
$\num([x,x+a))=a\cdot\num([0,1))$ for all rational numbers $a>0$.

\smallskip
\item
$\st\left(\frac{\num(X)}{\num([0,1))}\right)=\mu_L(X)$ for all $X\in\mathfrak{L}$.

\smallskip
\item
$\st\left(\frac{\num(X)}{\num([0,1))}\right)\le
\overline{\mu}_L(X)$ for all $X\subseteq\R$.
\end{enumerate}
\end{theorem}

\begin{proof}
Notice that the family of half-open intervals
$$\ring{I}\ =\ \left\{[x,x+a)\mid x\in\R \ \&\ a>0\right\}$$
generates a subalgebra $\B\subset\ring{L}$ whose nonempty sets have
all finite positive measure.
Then, by combining Theorems \ref{numeasure} and \ref{inner measure theorem},
we obtain the existence of an elementary numerosity $\num:\p(\R)\to[0,+\infty)_\F$
such that, for $\beta=\num([0,1))=\frac{\num([0,1))}{\mu_L([0,1))}$, one has:

\smallskip
\begin{enumerate}
\item[$(i)$]
$\num(X)=\num(Y)$ for all $X,Y\in\B$ with $\mu_L(X)=\mu_L(Y)$\,;

\smallskip
\item[$(ii)$]
$\num_\beta(X)=\mu_L(X)$ for all $X\in\mathfrak{L}$\,;

\smallskip
\item[$(iii)$]
$\num_\beta(X)\le\overline{\mu}_L(X)$ for all $X\subseteq\R$.
\end{enumerate}

\smallskip
Since $[x,x+a)\in\B$ for all $x\in\R$ and for all $a>0$, property $(1)$
directly follows from $(i)$. In order to prove $(2)$,
it is enough to show that $\num([0,a))=a\cdot\num([0,1))$ for all positive
$a\in\Q$. Given $p,q\in\N$, by $(1)$ and additivity we have that
$$\num\left(\left[0,\frac{p}{q}\right)\right)\ =\
\num\left(\bigsqcup_{i=0}^{p-1}\left[\frac{i}{q},\frac{i+1}{q}\right)\right)\ =\
\sum_{i=0}^{p-1}\num\left(\left[\frac{i}{q},\frac{i+1}{q}\right)\right)\ =\
p\,\cdot\num\left(\left[0,\frac{1}{q}\right)\right).$$
In particular, for $p=q$ we get that $\num([0,1))=q\cdot\num([0,1/q))$,
and hence property $(2)$ follows:
$$\num\left(\left[0,\frac{p}{q}\right)\right)\ =\
\frac{p}{q}\cdot\num\left([0,1)\right).$$

Finally, $(ii)$ and $(iii)$ directly correspond to properties
$(3)$ and $(4)$, respectively.
\end{proof}

\medskip
\begin{remark}
Let $\{X_n\mid n\in\N\}$ be a countable family of isometric,
pairwise disjoint, non-Lebesgue measurable sets such that the union
$A=\bigcup_{n\in\N} X_n$ is measurable with positive finite measure.
(\emph{E.g.}, one can consider a \emph{Vitali set}
on $[0,1)$ and take the countable family of its rational
translations modulo 1.)
Let $\num$ be an elementary numerosity as given by the
above theorem, and consider the finitely additive
measure $\num_\beta$ with $\beta=\num(A)/\mu(A)$.
Then, one and only one of the following holds:

\smallskip
\begin{itemize}
\item
$\num_\beta(X_n)=0$ for all $n\in\N$. In this case,
the measure $\num_\beta$ is not $\sigma$-additive
because $\num_\beta(A)=\mu_L(A)>0$.

\smallskip
\item
$\num_\beta(X_n)=\epsilon>0$ for some $n\in\N$.
In this case, $\num_\beta$ is not invariant with respect to
isometries, as otherwise one would get the contradiction
$\mu_{L}(A)=\num_\beta(A)\geq\sum_{n\in\N}\num_\beta(X_n)=
\sum_{n\in\N}\epsilon=+\infty$.
\end{itemize}
\end{remark}

\section{Numerosities and probability of infinite coin tosses}

The last application of elementary numerosities that we present
in this paper is about the existence of a non-Archimedean probability
for infinite sequences of coin tosses, which we propose as a sound
mathematical model for Laplace's original ideas.

\smallskip
Recall the \emph{Kolmogorovian framework}:

\smallskip
\begin{itemize}
\item
The \emph{sample space}
$$\Omega\ =\ \{H,T\}^\N\ =\ \left\{\omega\mid\omega:\N\to\{H,T\}\right\}$$
is the set of sequences which take either $H$ (``head") or $T$ (``tail") as values.

\smallskip
\item
A \emph{cylinder set} of codimension $n$ is a set of the form:\footnote
{~We agree that $i_1<\ldots<i_n$.}
$$C_{(t_1,\ldots,t_n)}^{(i_1,\ldots,i_n)}\ =\
\left\{\omega\in\Omega\mid\omega(i_s)=t_s\ \text{ for }s=1,\ldots,n\right\}$$
\end{itemize}

\smallskip
From the probabilistic point of view, the cylinder set
$C_{(t_1,\ldots,t_n)}^{(i_1,\ldots,i_n)}$ represents the event that
for every $s=1,\ldots,n$, the $i_s$-th coin toss gives $t_s$ as outcome.
Notice that the family $\ring{C}$ of all finite disjoint unions of cylinder sets
is an algebra of sets over $\Omega$.

\smallskip
\begin{itemize}
\item
The function $\mu_C:\ring{C}\to[0,1]$ is defined by setting:
$$\mu_C\left(C_{(t_1,\ldots,t_n)}^{(i_1,\ldots,i_n)}\right)\ = \ 2^{-n}$$
for all cylindrical sets, and then it is extended to a generic element of $\ring{C}$ by finite additivity:
$$
	\mu_C\left(C_{(t_1,\ldots,t_n)}^{(i_1,\ldots,i_n)} \cup \ldots \cup C_{(u_1,\ldots,u_m)}^{(j_1,\ldots,i_m)}\right)
	=
	\mu_C\left(C_{(t_1,\ldots,t_n)}^{(i_1,\ldots,i_n)} \right) + \ldots + \mu_C \left( C_{(u_1,\ldots,u_m)}^{(j_1,\ldots,i_m)}\right).
$$
\end{itemize}

\smallskip
It is shown that $\mu_C$ is a probability pre-measure on the ring $\ring{C}$.

\smallskip
Let $\A$ be the $\sigma$-algebra generated by the ring of
cylinder sets $\ring{C}$, 
and let $\mu:\A\to[0,1]$ be the unique probability measure that extends $\mu_C$,
as guaranteed by \emph{Caratheodory extension theorem}.
The triple $(\Omega,\A,\mu)$ is named the \emph{Kolmogorovian probability}
for \emph{infinite sequences of coin tosses}.

\smallskip
In \cite{nap} it is proved the existence of an elementary
numerosity $\num:\p(\Omega)\to[0,+\infty)_\F$ which is coherent
with the pre-measure $\mu_C$. Namely,
by considering the ratio $P(E)=\num(E)/\num(\Omega)$
between the numerosity of the given event $E$ and
the numerosity of the whole space $\Omega$, then
one obtains a \emph{non-Archimedean} finitely additive probability
$$P:\p(\Omega)\longrightarrow[0,1]_\F$$
that satisfies the following properties:

\smallskip
\begin{enumerate}
\item
If $F\subset\Omega$ is finite, then for all $E\subseteq\Omega$,
the conditional probability
$$P(E|F)\ =\ \frac{|E\cap F|}{|F|}.$$
\item
$P$ agrees with $\mu_C$ over all cylindrical sets:
$$P\left(C_{(t_1,\ldots,t_n)}^{(i_1,\ldots,i_n)}\right)\ =\
\mu_C\left(C_{(t_1,\ldots,t_n)}^{(i_1,\ldots,i_n)}\right)\ = \ 2^{-n}.$$
\end{enumerate}

\medskip
We are now able
to refine this result by showing that, up to infinitesimals,
we can take $P$ to agree with $\mu$ on the whole
$\sigma$-algebra $\A$.

\medskip
\begin{theorem}
Let $(\Omega,\A,\mu)$ be the Kolmogorovian probability
for infinite coin tosses. Then there exists an elementary numerosity
$\num:\p(\Omega)\rightarrow[0,+\infty)_\F$ such that the corresponding
non-Archimedean probability $P(E)=\num(E)/\num(\Omega)$
satisfies the above properties $(1)$ and $(2)$, along with the
additional condition:

\smallskip
\begin{enumerate}
\item[(3)]
$\st(P(E))=\mu(E)$ for all $E\in\A$.
\end{enumerate}
\end{theorem}

\begin{proof}
Recall that the family $\mathfrak{C}\subset\mathfrak{A}$ of
finite disjoint unions of cylinder sets is an algebra whose nonempty sets
have all positive measure. So, by applying Theorems
\ref{numeasure} and \ref{inner measure theorem},
we obtain an elementary numerosity $\num:\p(\Omega)\to[0,+\infty)_\F$
such that for every positive number of the form
$\beta=\frac{\num(Z)}{\mu(Z)}$ (where $0<\mu(Z)<+\infty$), one has:

\smallskip
\begin{enumerate}
\item[$(i)$]
$\num(C)=\num(C')$ whenever $C,C'\in\mathfrak{C}$ are
such that $\mu(C)=\mu(C')$\,;

\smallskip
\item[$(ii)$]
$\num_\beta(E)=\mu(E)$ for all $E\in\A$.
\end{enumerate}

\smallskip
Property $(1)$ trivially follows by recalling that elementary
numerosities of finite sets agree with cardinality:
$$P(E|F)\ =\ \frac{P(E\cap F)}{P(F)}\ =\
\frac{\frac{\num(E\cap F)}{\num(\Omega)}}{\frac{\num(F)}{\num(\Omega)}}\ =\
\frac{\num(E\cap F)}{\num(F)}\ =\ \frac{|E\cap F|}{|F|}.$$

Let us now turn to condition $(2)$.
Notice that for any fixed $n$-tuple of indices $(i_1,\ldots,i_n)$:

\begin{itemize}
\item
There are exactly $2^n$-many different $n$-tuples $(t_1, \ldots, t_n)$
of heads and tails;
\item
The associated cylinder sets $C_{(t_1,\ldots,t_n)}^{(i_1,\ldots,i_n)}$
are pairwise disjoint and their union equals the whole sample space $\Omega$.
\end{itemize}

By $(i)$, all those cylinder sets of codimension $n$
have the same numerosity
$\eta=\num\left(C_{(t_1,\ldots,t_n)}^{(i_1,\ldots,i_n)}\right)$
and so, by additivity, it must be $\num(\Omega)=2^n \cdot \eta$.
We conclude that
$$P\left(C_{(t_1, \ldots, t_k)}^{(i_1, \ldots, i_n)}\right)\ =\
\frac{\num\left(C_{(t_1, \ldots, t_k)}^{(i_1, \ldots, i_n)}\right)}
{\num(\Omega)}\ =\ \frac{\eta}{2^n\cdot\eta}\ =\ 2^{-n}.$$

We are left to prove $(3)$. By taking as
$\beta=\frac{\num(\Omega)}{\mu(\Omega)}=\num(\Omega)$,
property $(ii)$ ensures that for every $E\in\A$:
$$\mu(E)\ =\ \num_\beta(E)\ =\
\st\left(\frac{\num(E)}{\beta}\right)\ =\
\st\left(\frac{\num(E)}{\num(\Omega)}\right)\ =\
\st(P(E)).$$
\end{proof}

\end{document}